\documentclass[11pt]{amsart}
\usepackage{graphicx,pdflscape,rotating}
\usepackage{amsmath,epsfig,amsthm,amssymb,amsfonts,amsthm}
\usepackage[usenames]{color}
\usepackage{amssymb,amscd}
\usepackage{graphicx,tikz}
\usetikzlibrary{arrows,automata,chains,fit,shapes}
\newcommand{\Z}{\mathbb{Z}}
\newcommand{\N}{\mathbb{N}}

\newcommand{\R}{\mathbb{R}}

\newcommand{\GG}{\Gamma_d(q)}

\newcommand\g{\left( \begin{array}{cc} \Pi_{k=1}^{d-1} (t+l_k)^{m_k} & R \\ 0 & 1 \end{array} \right)}
\newcommand\m{(m_1,m_2, \cdots ,m_{d-1})}
\newcommand\LL{\mathcal L}

\newcommand{\MM}{M_{n,\epsilon}}

\newtheorem{theorem}{Theorem}

\newtheorem{lemma}[theorem]{Lemma}
\newtheorem{prop}[theorem]{Proposition}

\newtheorem{defn}{Definition}

\newcommand{\ou}{\overline{u}}
\newcommand{\ov}{\overline{v}}

\title{Higher rank lamplighter groups are graph automatic}

\author[Sophie B\'{e}rub\'{e}] {Sophie B\'{e}rub\'{e}}
\address{Department of Biostatistics, Johns Hopkins Bloomberg School of Public Health, Baltimore, MD 21205}
\email{sberube3@jhmi.edu}

\author[Tara Palnitkar] {Tara Palnitkar}
\address{Department of Mathematics, University of Minnesota, Minneapolis, MN 55455}
\email{palni003@umn.edu}

\author[Jennifer Taback] {Jennifer Taback}
\address{Department of Mathematics, Bowdoin College, Brunswick, ME 04011}
\email{jtaback@bowdoin.edu}

\thanks{The third author acknowledges support from
National Science Foundation grant DMS-1105407 and Simons Foundation grant 31736 to Bowdoin College. The authors would like to thank Murray Elder for helpful conversations during the writing of this paper.}


\date{\today}
\begin{document}

\begin{abstract}
We show that the higher rank lamplighter groups, or Diestel-Leader groups $\Gamma_d(q)$ for $d \geq 3$, are graph automatic.   As these are not automatic groups, this introduces a new family of graph automatic groups which are not automatic.
\end{abstract}

\maketitle

\section{Introduction}

Automatic groups were introduced in \cite{WordProc} by Cannon, Gilman, Epstein, Holt and Thurston, motivated by initial observations of Cannon and Thurston about hyperbolic groups and their geometry.  Their goal was first to understand fundamental groups of compact 3-manifolds, and then to streamline computation in these groups.
For example, if $G$ is an automatic group, then its Dehn function is at most quadratic, the word problem can be solved in quadratic time, and the automatic structure can be used to reduce any word in the generating set to a normal form for that group element, also in quadratic time.

A finitely generated group $G$ has an automatic structure with respect to a generating set $S$ if there is a regular language of normal forms for elements of $G$, and, for each $s \in S$, a finite state automaton which recognizes multiplication by $s$.  The class of automatic groups includes all finite groups, braid groups, Coxeter groups, hyperbolic groups and mapping class groups, among others. It is shown in \cite{WordProc} that if $G$ has an automatic structure with respect to one generating set, then it has an automatic structure with respect to any generating set.  All automatic groups are finitely presented, and there are geometric conditions which can be used to show that a set of normal forms constitutes the basis of an automatic structure.  For a comprehensive introduction to automatic groups, see \cite{WordProc}, or for a shorter treatment,  \cite{farb}, \cite{holt} or \cite{NS}.

It is unsatisfying that many groups which have nice algorithmic properties, including the properties listed above, are not automatic.  For instance, a finitely generated nilpotent group is automatic if and only if it is virtually abelian.  In \cite{KKM},  Kharlampovich, Khoussainov, and Miasnikov extend the definition of an automatic group to a {\em (Cayley) graph automatic group}, in which the language of normal forms representing group elements is defined over a finite alphabet of symbols.  If one takes the symbol alphabet to be the generating set for the group, then the definition of an automatic group is recovered.  Graph automatic groups retain many of the computational advantages of automatic groups, and this enlarged class includes the solvable Baumslag-Solitar groups $BS(1,n)$, the lamplighter groups $\Z_n \wr \Z$, the metabelian groups $\Z^n \ltimes_A \Z$ for $A \in SL_n(\Z)$, and all finitely generated groups of nilpotency class at most $2$, among others. \cite{KKM}  It is shown in \cite {BK} that the non-solvable Baumslag-Solitar groups $BS(m,n)$ are also graph automatic.

In this paper, prove the following theorem.
\begin{theorem}
\label{thm:main}
The Diestel-Leader groups $\GG$ for $d \geq 3$ are graph automatic.
\end{theorem}

The family of Diestel-Leader groups $\Gamma_d(q)$ for $d \geq 3$, or higher rank lamplighter groups, was introduced formally in \cite{BNW} by Bartholdi, Neuhauser and Woess.  These groups are not automatic, as they are type $F_{d-1}$ but not $F_d$ when $d \geq 3$, and automatic groups are of type $FP_{\infty}$.  These groups are defined explicitly in Section \ref{sec:DL} below, and their metric properties are studied by the third author and Stein in \cite{ST}.  Kevin Wortman has sketched a proof showing that arguments analogous to those of Gromov in \cite{G} imply that  the Dehn function of $\GG$ is quadratic regardless of the values of $d \geq 3$ and $q$.  It was shown in \cite{CT} that when $p$ is prime,  $\Gamma_3(p)$ is a cocompact lattice in $Sol_5 \left({\mathbb F}_p((t))
\right)$, and its Dehn function is quadratic.  The Dehn function of $\Gamma_3(m)$ is studied for any $m$  in \cite{KR}
where it is shown to be at most quartic.

The Cayley graph of the Diestel-Leader group $\GG$, with respect to a certain generating set is a {\em Diestel-Leader graph}, a particular subset of a product of $d$ infinite trees of valence $q+1$.  More general Diestel-Leader graphs were introduced in \cite{DL}
as a potential answer to the question ``Is any connected, locally finite, vertex transitive graph quasi-isometric to the Cayley graph of a
finitely generated group?"   The Diestel-Leader graph which is a subset of a product of two infinite trees of differing valence is not quasi-isometric to the Cayley graph of any finitely generated group, as shown by Eskin, Fisher and Whyte in \cite{EFW}.  The Cayley graph of the well-known lamplighter group $L_q = \Gamma_2(q) = \Z_q \wr \Z$, with respect to a natural generating set, is the Diestel-Leader graph contained in a product of two infinite trees of valence $q+1$.  In this sense we view the Diestel-Leader groups as a geometric higher rank generalization of the lamplighter groups.

\section{Diestel-Leader Groups}
\label{sec:DL}

We briefly introduce the Diestel-Leader groups $\GG$ and their geometry, and refer the reader to \cite{BNW}, \cite{ST} and \cite{STW} for a more comprehensive treatment.  The Diestel-Leader graph $DL_d(q)$ is the subset of the product of $d$ infinite regular trees $T_1,T_2, \cdots ,T_d$, each with valence $q+1$ and a height function $h_i:T_i \rightarrow \R$, consisting of the vertices for which the sum of the heights of the coordinates is equal to zero.  Two vertices are connected by an edge if and only if they are identical in all but two coordinates, and in those two coordinates, say $i$ and $j$, the entries differ by an edge in $T_i$ or $T_j$.

Bartholdi, Neuhauser and Woess in \cite{BNW} present a matrix group $\GG$ with a particular generating set $S_d(q)$  so that the Cayley graph $\Gamma(\GG,S_d(q))$ is exactly the Diestel-Leader graph $DL_d(q)$.  Their construction relies on the arithmetic condition that $d-1<p$ for all prime divisors $p$ of $q$. Specifically, let $\mathcal{L}_q$ be a commutative ring of order $q$ with multiplicative
unit 1, and suppose $\mathcal{L}_q$ contains distinct elements $l_1, \dots, l_{d-1}$ such that if $d\geq 3$, their pairwise differences are invertible.  In this paper, we additionally assume that $l_i$ is also invertible, for $1 \leq i \leq d$.  These conditions are easily satisfied, for example, in $\Z_q$ for large enough $q$.

Define a ring of polynomials in the formal variables $t$ and $(t+l_i)^{-1}$ for $1 \leq i \leq d-1$ with
finitely many nonzero coefficients lying in ${\mathcal L}_q$:
$${\mathcal R}_d({\mathcal L}_q) = {\mathcal L}_q[t,(t+l_1)^{-1},(t+l_2)^{-1}, \cdots ,(t+l_{d-1})^{-1}].$$

The Diestel-Leader group constructed in \cite{BNW} is the group of affine matrices of the
form
\begin{equation}\label{eqn:g}
\left( \begin{array}{cc} (t+l_1)^{m_1} \cdots (t+l_{d-1})^{m_{d-1}} & P \\ 0 & 1 \end{array} \right), \text{ with }
m_1,m_2, \cdots ,m_{d-1} \in \Z \text{ and }P \in {\mathcal R}_d({\mathcal L}_q),
\end{equation}
which has Cayley graph $DL_d(q)$ with respect to the generating set $S_{d,q}$ consisting of the matrices
$$\left( \begin{array}{cc} t+l_i & b \\ 0 & 1 \end{array} \right)^{\pm 1}, \text{ with } b \in {\mathcal L}_q, \ i \in
\{1,2, \cdots ,d-1\} \text{ and }$$
$$ \left( \begin{array}{cc} (t+l_i)(t+l_j)^{-1} & -b(t+l_j)^{-1} \\ 0 & 1 \end{array} \right), \text{ with } b \in {\mathcal
L}_q, \ i,j \in \{1,2, \cdots ,d-1\}, \ i \neq j.$$

We refer to a matrix of the form  $\left( \begin{array}{cc} t+l_i & b \\ 0 & 1 \end{array} \right)$ as a {\em type 1} generator and a matrix of the form $\left( \begin{array}{cc} (t+l_i)(t+l_j)^{-1} & -b(t+l_j)^{-1} \\ 0 & 1 \end{array} \right)$ as a {\em type 2} generator.  Without loss of generality, we assume that in a type 2 generator, $i<j$.

An element $g \in \GG$ is uniquely defined by a $(d-1)$-tuple of integers $(m_1,m_2, \cdots ,m_{d-1})$ which determine the upper left entry of $g$ and a polynomial $P \in {\mathcal R}_d({\mathcal L}_q)$.  The identification between a group element and a vertex in the Diestel-Leader graph $DL_d(q)$ is based on this information as well, and is explicitly described in \cite{BNW} as well as \cite{STWarxiv}, the extended version of \cite{STW}.  Roughly, each tree is associated with one of the defining variables in ${\mathcal R}_d({\mathcal L}_q)$, and vertices in that tree are assigned equivalence classes of polynomials in that variable.  To find the coordinate in the tree $T_i$ corresponding to $g \in \GG$ of the above form, we compute the Laurent polynomial
$${\mathcal LS}_i((t+l_1)^{-k_1} \cdots (t+l_{d-1})^{-k_{d-1}} P)$$
and consider only those terms of negative degree, or non positive degree when $i=d$.  To show this is well defined we refer to the following lemma, proven in \cite{STW}.  The proof relies on rewriting $R$ as a Laurent polynomial in each of the possible variables.

\begin{lemma}[Decomposition Lemma]
\label{lemma:decomp}
Let $$Q \in {\mathcal R}_d({\mathcal L}_q)={\mathcal L}_q [ (t+l_1)^{-1},(t+l_2)^{-1}, \cdots , (t+l_{d-1})^{-1},t ]$$   where the $l_i \in {\mathcal L}_q$ are chosen so that $l_i-l_j$ is invertible whenever $i \neq j$. Then $Q$ can be written uniquely as $P_1(Q) + P_2(Q) + \cdots + P_d(Q)$ where
\begin{enumerate}
\item[(a)] for $1 \leq i \leq d-1$ we have that $P_i(Q)$ is a polynomial in $t+l_i$ all of whose terms have negative degree, and
\item[(b)] for $i=d$ we have that $P_d(Q)$ is a polynomial in $t^{-1}$ all of whose terms have non-positive degree.
\end{enumerate}
\end{lemma}

While we will not explicitly use the identification between a group element and the corresponding vertex in the Cayley graph in this paper, we will use the Decomposition Lemma repeatedly.  That is, for $g \in \GG$ as in Equation \eqref{eqn:g}, we will decompose a polynomial related to $P$ using the Decomposition Lemma, and use the component polynomials as the basis of our graph automatic structure.  In our analysis of the relationship between $gs$ and $g$, for $s \in S_d(q)$, we expand the upper right entry of the product $gs$ using the same techniques.

It is easily verified that the following combinatorial formulae allow us to rewrite polynomials in $t+l_r$ in terms of $t+l_s$ for $i \neq j$, and $t^{-1}$:
\begin{equation}\label{eqn:1}
(t+l_r)^k = \sum_{n=0}^{\infty} { k \choose n} (l_r-l_s)^{k-n}(t+l_s)^n
\end{equation}
and
\begin{equation}\label{eqn:2}
(t+l_r)^k =  \sum_{n=-k}^{\infty} { k \choose - n}l_i^{n+k}t^{-n}
\end{equation}
for any $k \in \Z$.  Moreover, when $k$ is nonnegative, we write $t^k$ as a polynomial in $t+l_s$ as follows:
\begin{equation}\label{eqn:3}
t^k = \sum_{n=0}^{k} {k \choose n} (-l_s)^{k-n}(t+l_s)^n.
\end{equation}

In the proofs below, we repeatedly use Equations \eqref{eqn:1} and \eqref{eqn:2} with a fixed value of $r \in \{1,2, \cdots ,d-1\}$ and $k=-1$ to rewrite $(t+l_r)^{-1}$ in terms of $t+l_s$ or $t^{-1}$.  In the first case, write
\begin{equation}\label{eqn:1a}
(t+l_r)^{-1} = \sum_{n=0}^{\infty} \alpha_n (t+l_s)^n
\end{equation}
where $\alpha_n = {-1 \choose n} (l_r-l_s)^{-(n+1)}=(-1)^n (l_r-l_s)^{-(n+1)}$.  Writing $C_{r,s} =(l_r-l_s)^{-1}$ for $1 \leq s \leq d-1$, this simplifies to $\alpha_n = (-1)^n C_{r,s}^{n+1}$.  Notice that for fixed values of $r$ and $s$, we have $\alpha_{n+1} = -C_{r,s} \alpha_n$.

Similarly, we simplify Equation \eqref{eqn:2} with exponent $-1$ as
\begin{equation}\label{eqn:2a}
(t+l_r)^{-1} = \sum_{n=1}^{\infty} { -1 \choose - n} (l_r)^{n-1}t^{-n} = \sum_{n=1}^{\infty}(-1)^{n-1} (l_r)^{n-1}t^{-n}.
\end{equation}
Denote the coefficients in the above sum as $\alpha_n'=(-1)^{n-1} (l_r)^{n-1}$, and note that $\alpha_{n+1}'=-l_r \alpha_n'$.

We make one assumption about our Diestel-Leader groups to simplify notation throughout this paper.  Namely we take $\LL_q = \Z_q$ and note that all our theorems hold for more general coefficient rings as well.

\section{Graph Automatic Groups}
\label{sec:automatic}

Let $G$ be a group with finite symmetric  generating set $X$, and $\Lambda$ a finite set of symbols.
The number of symbols (letters) in a word $u\in\Lambda^*$ is denoted $|u|_{\Lambda}$. We begin by defining a convolution of group elements, following \cite{KKM} in our notation.

\begin{defn}[convolution]
Let $\Lambda$ be a finite set of symbols, $\diamond$  a symbol not in $\Lambda$,  and let $L_1,\dots, L_k$ be a finite set of languages over $\Lambda$. Set  $\Lambda_{\diamond}=\Lambda\cup\{\diamond\}$. Define the {\em convolution of a tuple} $(w_1,\dots, w_k)\in L_1\times \dots \times L_k$ to be the string $\otimes(w_1,\dots, w_k)$ of length $\max |w_i|_{\Lambda}$ over the alphabet $\left(\Lambda_{\diamond}\right)^k$  as follows.
The $i$th symbol of the string is
\[\left(\begin{array}{c}
\lambda_1\\
\vdots\\
\lambda_k
\end{array}\right)\]
where $\lambda_j$ is the $i$th letter of $w_j$ if $i\leq |w_j|_{\Lambda}$ and $\diamond$ otherwise.
Then  \[\otimes(L_1,\dots, L_k)=\left\{\otimes(w_1,\dots, w_k) \mid w_i\in L_i\right\}.\]
\end{defn}
We note that the convolution of regular languages is again a regular language.

As an example, if $w_1=abb, w_2=bbb$ and $w_3=ba$ then
\[\otimes(w_1,w_2,w_3)=\left(\begin{array}{c}
a\\
b\\
b
\end{array}\right)
\left(\begin{array}{c}
b\\
b\\
a
\end{array}\right)
\left(\begin{array}{c}
b\\
b\\
\diamond
\end{array}\right)\]

When $L_i=\Lambda^*$ for all $i$ the exact definition in \cite{KKM} is recovered.

The definition of a graph automatic group extends that of an automatic group by allowing the normal forms for group elements to be defined over a finite alphabet of symbols.  When this set of symbols is simply taken to be the set of group generators, the definition of an automatic group is recovered.

A set of normal forms for a group may additionally be quasi-geodesic, defined as follows.
\begin{defn}[quasigeodesic normal form]
A  {\em normal form for $(G,X,\Lambda)$} is a set of words $L\subseteq \Lambda^*$ in bijection with $G$. A normal form $L$ is
  {\em quasigeodesic}  if there is a constant $D$ so  that $$|u|_{\Lambda}\leq D(||u||_X+1)$$ for each $u\in L$,  where $||u||_X$ is the length of a geodesic  in $X^*$ for the group element represented by  $u$.
\end{defn}
The $||u||_X+1$ in the definition allows for normal forms where the identity of the group is represented by a nonempty string of length at most $D$.  We denote the image of $u\in L$ under the bijection with $G$ by $\ou$.

The following definition was introduced in \cite{KKM}.

\begin{defn}[graph automatic group] \label{def:graph}
Let $(G,X)$ be a group and finite symmetric generating set, and $\Lambda$ a finite set of symbols.  We say that $(G,X,\Lambda)$ is {\em graph automatic} if there is a regular normal form $L \subset \Lambda^*$, such that for each $x\in X$ the language $L_x=\{\otimes(u,v) \mid u,v\in L,  \ov =_G \ou x\}$ is a regular language.
\end{defn}
The language $L_x$ is often referred to as a {\em multiplier language}.

Any set of normal forms forming the basis of an automatic structure for a group $G$ is automatically quasigeodesic.  The proof of Lemma 8.2 of \cite{KKM} contains  the  observation that graph automatic groups naturally possess a quasigeodesic normal form, and a proof is included in \cite{ETC}.

\begin{lemma}\label{lem:linearnf}Let $G$ be a group with finite generating set $X$.
If $(G,X,\Lambda)$ is graph automatic with respect to the regular normal form $L$, then $L$ is a quasigeodesic normal form.
\end{lemma}
The existence of a quasigeodesic regular normal form in an automatic or graph automatic structure ensures that the word problem is solvable in quadratic time.  While the Diestel-Leader groups are metabelian, and hence have solvable word problem,  we note that it is a simple consequence of Theorem 9 of \cite{ST}, stated below as Theorem \ref{thm:ST}, that the set of normal forms defined in Section \ref{sec:nf} is quasi-geodesic.  We prove this in Section \ref{sec:nf}.

We conclude with two straightforward lemmas about convolutional languages which we will refer to in the verification of our graph automatic structure for $\GG$.

\begin{lemma}[Offset Lemma]\label{lemma:offset}
Let ${\mathcal A}$ and ${\mathcal B}$ be regular languages, with
$${\mathcal A'}=\{ \otimes(a_1,a_2) | a_i \in {\mathcal A}, |a_1|=|a_2|\}$$
and ${\mathcal B'}$ any subset of $\otimes({\mathcal B},{\mathcal B})$ which is a regular language.  Let $\Lambda$ be any finite alphabet.  Then
$$\{ \otimes(a_1b_1,a_2xb_2),\otimes(a_1yb_1,a_2b_2) | \otimes(a_1,a_2) \in {\mathcal A}', \otimes(b_1,b_2) \in {\mathcal B}', x,y \in \Lambda \}$$
is a regular language.
\end{lemma}

The proof of Lemma \ref{lemma:offset} follows easily from the next lemma, whose proof is given in \cite{TY}.
\begin{lemma}
\label{lemma:shift}
Let ${\mathcal L}$ be a regular language defined over a finite alphabet $\Lambda$.   Then the set $$\{\otimes(xw,w) | w \in {\mathcal L}, \ x \in \Lambda\}$$ forms a regular language.
\end{lemma}

\section{A regular language of normal forms}
\label{sec:nf}

We begin the construction of a graph automatic structure for $\GG$ with a quasigeodesic normal form which is also a regular language.   We present a short proof that our normal form language is quasigeodesic, as it follows nicely from the computation of the word metric for these groups given in \cite{ST}.

Let $g=\g$.  The vector ${\bf m}=\m$ and the polynomial $R$ uniquely define $g$.  However, $g$ is also uniquely defined from the data ${\bf m}$ and the polynomial $$R'=\Pi_{m=1}^{d-1} (t+l_i)^{-m_i}R$$ and this forms the basis of the normal form we use for our regular language.  Namely, using the Decomposition Lemma from \cite{ST} (Lemma \ref{lemma:decomp}) we can uniquely decompose $R'$ as follows: $$R' = R_1 + R_2 + \cdots R_d$$ where $R_i$ for $1 \leq i < d$ contains only terms in the formal variable $t+l_i$ of negative degree, and $R_d$ contains terms in the variable $t^{-1}$ of nonpositive degree.

As concatenations of regular languages are regular, we define a prefix language and a suffix language for our normal form which are both regular, and whose union gives a regular language of normal forms for elements of $\GG$. Let the prefix language ${\mathcal P}$ be defined over the alphabet $\{x,y,\#\}$, and encode the vector $\m$ as $\epsilon_1^{m_1} \# \epsilon_2^{m_2} \# \cdots \# \epsilon_{d-1}^{m_{d-1}}$, where $\epsilon_i=x$ if $m_i > 0$ and $\epsilon_i=y$ if $m_i < 0$; if $m_i=0$ we omit $\epsilon_i$.  Note that the prefix corresponding to the identity is $\#^{d-2}$.  It is clear that ${\mathcal P}$ is a regular language.

We now encode the information contained in the polynomials $R_1, \cdots ,R_d$ as a suffix language ${\mathcal S}$ over the alphabet $\{\#,b_0,b_1, \cdots ,b_{q-1}\}$ where $\{b_0,b_1, \cdots ,b_{q-1}\} =  \Z_q$.   If $R_i \neq 0$, denote its minimal degree $-\delta_i$, for $\delta_i \in \N$, and when $i \neq d$ write $$R_i= \sum_{j=1}^{\delta_i} \beta_{i,j} (t+l_i)^{-j},$$ including a coefficient of $0$ when $(t+l_i)^{-n}$ is not present in $R_i$.  When $i=d$ we include a constant term in the polynomial expression.  With the given indexing, we allow $\beta_{i,1}=0$ but $\beta_{i ,\delta_i} \neq 0$.  If $R_i=0$ then let $\delta_i=0$ as well.

Let $S_i = \beta_{1} \beta_2 \cdots \beta_{\delta_i}$ with $\beta_j \in \Z_q$ for $1 \leq i \leq d$ denote the the string of coefficients of $R_i$, with $S_i = \emptyset$ if $R_i=0$, the entry in the suffix language ${\mathcal S}$  corresponding to the tuple $R_1, \cdots ,R_d$ is
$$S_1 \# S_2 \# \cdots \# S_d.$$
We use the string $\#^{d-1}$ to denote the suffix string when $R_1=R_2= \cdots =R_d=0$.  Let ${\mathcal S}$ be the union of all strings of the above form; it is clear that ${\mathcal S}$ is a regular language and hence the language of concatenations ${\mathcal N} = {\mathcal P}{\mathcal S}$ describes a regular language of normal forms for elements of $\GG$.  Elements of the normal form language ${\mathcal N}$ will be written as $(p,s)$ for some $p \in {\mathcal P}$ and $s \in {\mathcal S}$.  If $u \in {\mathcal N}$, let $\bar{u}$ denote the corresponding element of $\GG$, and if $g \in \GG$ we let $\nu(g)$ denote the corresponding element of ${\mathcal N}$, with $\pi(g)$ and $\sigma(g)$, respectively, denoting the prefix and suffix strings of  $\nu(g)$.

In Theorem \ref{thm:nf-lang} below we show that this normal form language is quasigeodesic.  We will use the following theorem from \cite{ST} which relates the word metric with respect to $S_d(q)$ with the product metric on the product of trees underlying the Diestel-Leader graph $DL_d(q)$.

\begin{theorem}[\cite{ST}, Thm. 9]\label{thm:ST}
Let $l(g)$ denote the word length of $g \in \GG$ with respect to the generating set $S_{d,q}$
and $d_T (g)$ the distance in the product metric on the product of trees between the vertex in $DL_d(q)$ corresponding to $g$  and $o$,
the fixed basepoint corresponding to the identity in $\GG$. Then
$$\frac{1}{2} d_T(g) \leq l(g) \leq 2d_T(g)$$
that is, the word length is quasi-isometric to the distance from the identity in the product metric
on the product of trees.
\end{theorem}

The word metric in $\GG$ is computed explicitly in \cite{ST} and we refer the reader to that paper for more explanation of the facts given below.

Let $o=(o_1,o_2, \cdots ,o_d)$ denote the vertex of $DL_d(q)$ corresponding to the identity in $\GG$.  Let $g \in \GG$ correspond to $(\tau_1,\tau_2, \cdots ,\tau_d) \in DL_d(q)$, where $\tau_k$ is a vertex in the $k$-th tree $T_k$ in the product.  Let $u_k=d(o_k,o_k\curlywedge \tau_k)$ and $v_k=d(\tau_k,o_k\curlywedge \tau_k)$.
\begin{itemize}
\item As they represent distances in a tree, $0 \leq u_k$ and $0 \leq v_k$ for all
$k$.

\smallskip

\item In $T_k$, the shortest path from $o_k$ to $\tau_k$ has length $u_k+v_k$.  This path necessarily looks either like an inverted vee, in which case the length of the ``upward" portion has length $u_k$ and the ``downward" portion has length $v_k$, or $u_k=0$ and the path is a single segment of length $v_k$.

\smallskip

\item If $h_k$ is the height function defined on $T_k$, then $h_k(\tau_k)=v_k-u_k$.  Moreover, if $g$ is given as in Equation \eqref{eqn:g}, and $(m_1,m_2, \cdots ,m_d)$ is the $d$-tuple of exponents in the upper left entry of the matrix for $g$, then the height of $\tau_k$ in $T_k$ is $m_k$, so $m_k=v_k-u_k$.

\smallskip

\item The defining conditions of the Diestel-Leader graph ensure that $\sum_{i=1}^d (v_i-u_i) = 0$.
\end{itemize}

Using this notation we define a projection from $\GG$ to $(\Z^2)^d$ by
$$\Pi(g)= \Pi((\tau_1,\tau_2, \cdots ,\tau_d)) = \left((u_1,v_1),(u_2,v_2), \cdots ,(u_d,v_d) \right).$$  We will use this projection in the proof of Theorem \ref{thm:nf-lang}.

\begin{theorem}\label{thm:nf-lang}
The language ${\mathcal N} = {\mathcal P} {\mathcal S}$ of normal forms for elements of $\GG$ for $d \geq 3$ described above is a regular quasigeodesic language.
\end{theorem}

\begin{proof}
Above we showed that ${\mathcal N}$ is a regular language; it remains to prove that it is quasigeodesic as well.
Let  $g=\g \in \GG$ and $u_i$ and $v_i$ be as above.  Then $$d_T(g) = \sum_{j=1}^d u_j + \sum_{j=1}^d v_j = 2 \sum_{j=1}^d u_j = 2\sum_{j=1}^d v_j.$$

The length of $\nu(g)$ is
$$|\nu(g)| = \sum_{j=1}^{d-1} |m_j| + (d-2) + \sum_{j=1}^d \delta_j + 1+(d-1),$$ where $-\delta_i$ is the minimal degree of the polynomial $R_i$ defined above, the $d-2$ and $d-1$ summands correspond to the number of $\#$ symbols in the prefix or suffix string of the normal form, and the extra $1$ reflects the fact that only $R_d$ has a constant term.

When computing the tree distance $d_T$, $u_i$ is the minimal degree of ${\mathcal LS}_i(R)$, if this degree is negative, and zero otherwise, as described in \cite{STWarxiv}.  To determine $\nu(g)$ we compute ${\mathcal LS}_i(\Pi_{j=1}^{d-1} (t+l_j)^{-m_j} R)$ and denote its minimal degree by $-\delta_i$, for $\delta_i >0$ ($\geq 0$ if $i=d$), and zero otherwise.  The factor of $(t+l_i)^{-m_i}$ in the above expression relates $u_i$ and $\delta_i$, namely,  $\delta_i=u_i-m_i$. Hence

\begin{align*}
|\nu(g)| &= \sum_{j=1}^{d-1} |m_j| + \sum_{j=1}^d (u_j-m_j) + 2d-2 \\
 &= \sum_{j=1}^{d-1} |v_j-u_j| + 2\sum_{j=1}^d u_j  -\sum_{j=1}^d v_j  + 2d-2 \\
 &=  \sum_{j=1}^{d-1} |v_j-u_j| + \sum_{j=1}^d u_j   + 2d-2
\end{align*}
where the last equality holds because $\sum_{j=1}^d u_j=\sum_{j=1}^d v_j$.
As $u_i,v_i \geq 0$, we make the comparisons:
$$|\nu(g)| \leq \sum_{j=1}^{d-1} v_j + \sum_{j=1}^{d-1} v_j + \sum_{j=1}^d u_j   + 2d-2  \leq 3\sum_{j=1}^d u_j  + 2d-2 \leq 2 d_T(g) + 2d-2$$
and
$$|\nu(g)| \geq \sum_{j=1}^d u_j + 2d-2 \geq \frac{1}{2} d_T(g) + 2d-2$$
from which it follows that ${\mathcal N}$ is a quasigeodesic regular language.
\end{proof}

\section{Multiplier languages}

We now show that the multiplier languages $\LL_s$, where $s \in S_d(q)$, arising from this normal form ${\mathcal N}$ are also regular.  The multiplier language $\LL_s$ consists of convolutions $\otimes(\nu(g),\nu(gs))$ where $g \in \GG$.  It suffices to check that $L_s$ is regular when $s$ has one of two forms:
$$s \in \left\{  \left( \begin{array}{cc}  t+l_i & b \\ 0 & 1 \end{array} \right),  \left( \begin{array}{cc} (t+l_i)(t+l_j)^{-1} & b(t+l_j)^{-1} \\ 0 & 1 \end{array} \right) \right\}$$
with $b \in \Z_q$, as it is proven in \cite{ET} that $L_s$ is a regular language if and only if $L_{s^{-1}}$ is a regular language.   The following theorem constructs the remainder of the graph automatic structure for $\Gamma_d(q)$.

\begin{theorem}\label{thm:mult-lang}
Let $s \in S_d(q)$ be a
\begin{itemize}
\item a type 1 generator of $\GG$ of the form $\left( \begin{array}{cc}  t+l_i & b \\ 0 & 1 \end{array} \right)$, or
\item a type 2 generator of the form $\left( \begin{array}{cc} (t+l_i)(t+l_j)^{-1} & b(t+l_j)^{-1} \\ 0 & 1 \end{array} \right)$
\end{itemize}
where $1 \leq i,j \leq d-1$, $i < j$ and $b \in \Z_q$. The language ${\mathcal L}_s = \{ \otimes(\nu(g),\nu(gs)) | g \in \GG \}$ is a regular language.
\end{theorem}

We prove Theorem \ref{thm:mult-lang} in several steps.  First observe that the relationship between the prefix strings $\pi(g)=(m_1,m_2, \cdots ,m_{d-1})$ and $\pi(gs)=(m'_1,m'_2, \cdots ,m'_{d-1})$ is easily determined:

\begin{enumerate}
\item if $s$ is a type 1 generator, $m'_i=m_i+1$ and $m'_k = m_k$ for all other $k$.

    \medskip

\item if $s$ is a type 2 generator and $i<j$, then $m'_i=m_i+1$ and $m'_j=m_j-1$.  For all $k \neq i,j$ we have $m'_k=m_k$.
\end{enumerate}
These conditions are easily checked with a finite state machine, and we conclude that for a fixed generator, the set of strings
${\mathcal P}_s = \{ \otimes(p_1,p_2) | p_i \in {\mathcal P} \}$ satisfying the above conditions is a regular language, regardless of whether these prefix strings arise from $g$ and $gs$.

The next step in the proof of Theorem \ref{thm:mult-lang} is to determine the relationship between $\sigma(g)$ and $\sigma(gs)$, and construct a finite state machine which recognizes this relationship between any two suffix strings.  Beginning with a pair $g,gs \in \GG$, if the polynomials $R_1,R_2, \cdots R_d$ determine $\sigma(g)$ and $Q_1,Q_1, \cdots ,Q_d$ determine $\sigma(gs)$, we can write the suffix strings as  $\sigma(g) = S_1 \# S_2 \# \cdots \# S_d$ and $\sigma(gs) = S'_1 \# S'_2 \# \cdots \# S'_d$ where each $S_i=\sigma(R_i)$ and $S_i'=\sigma(Q_i)$ is a string of elements of $\Z_q$.  In Sections \ref{sec:aut1} and \ref{sec:aut2} we determine the algebraic relationship between the entries of $S_k$ and $S_k'$, and build a finite state machine which recognizes this relationship.  Namely, for a fixed generator $s$ we first construct finite state automata which verify:
\begin{enumerate}
\item for each $n \neq i$, that the strings $\sigma(R_n)$ and $\sigma(Q_n)$ differ in the appropriate manner.  This step creates $d-1$ finite state automata.
\item when $n=i$, that the first entry in  $\sigma(Q_i)$ relates to the remaining coefficients in $\sigma(g)$ in the appropriate manner.  This entry corresponds to the coefficient of $(t+l_i)^{-1}$ in $Q_i$.
\item when $n=i$ that the remaining entries in $\sigma(R_i)$ and $\sigma(Q_i)$ differ in appropriate manner.
\end{enumerate}
These finite state automata are combined to accept a regular language of convolutions satisfying all of the above conditions, which includes $\{\otimes(\sigma(g),\sigma(gs)) | g \in \GG \}$.  Let $\mathrm{M}_s$ denote the automaton which accepts this language.

In Lemmas \ref{lemma:lengthsame1} and \ref{lemma:lengthsame2} below we determine the relative lengths of $S_k$ and $S'_k$ arising in this way, when $s$ is first a type 1 generator and then a type 2 generator.

Let $\overline{{\mathcal N}'_s} = \otimes ({\mathcal S},{\mathcal S}) = \{\otimes(\sigma(g),\sigma(h)) | g,h \in \GG \}$  be the language of convolutions of all possible suffix strings arising from elements of $\GG$.  Generically, we view an element of $\overline{{\mathcal N}'_s}$ as follows.  Let $\Phi_n$ and $\Psi_n$ be strings of length $\eta_n$ and $\chi_n$, respectively, of symbols from the finite alphabet consisting of elements of $\Z_q$, for $1 \leq n \leq d$.  Then $\otimes(\Phi_1\#\Phi_2\#\cdots\# \Phi_d,\Psi_1\#\Psi_2\#\cdots \#\Psi_d) \in \overline{{\mathcal N}'_s}$.  Without loss of generality, for $n \neq i$ we assume (as these conditions can be verified with finite state automata) that the lengths $\eta_k$ and $\chi_k$ agree with Lemma \ref{lemma:lengthsame1} if $s$ is a type 1 generator and  Lemma \ref{lemma:lengthsame2} if $s$ is a type 2 generator.  It is clear that $\overline{{\mathcal N}'_s}$ is a regular language.

Using the finite state automaton $\mathrm{M}_s$ constructed to accept convolutions of the form $\otimes(\sigma(g),\sigma(gs))$, we conclude that the subset ${\mathcal N}'_s$ of $\overline{{\mathcal N}'_s}$ of convolutions which are accepted by $\mathrm{M}_s$ is also a regular language.  In this language, $\Phi_n$ and $\Psi_n$ have the same relationship as if they arose from $\sigma(g)$ and $\sigma(gs)$, respectively.  Define ${\mathcal N}_s$ to be the language of concatenations ${\mathcal P}_s {\mathcal N}_s'$.  The following proposition follows immediately.

\begin{prop}\label{prop:concat}
The language of concatenations ${\mathcal N}_s={\mathcal P}_s {\mathcal N}_s'$ is a regular language.
\end{prop}

To conclude the proof of Theorem \ref{thm:mult-lang} we show in Section \ref{sec:aut2} that ${\mathcal N}_s = {\mathcal L}_s$.

\section{Construction of automata I}
\label{sec:aut1}

Let $s$ be either a type 1 or type 2 generator, so values of $i$ and possibly $j$ are fixed.  In this section we determine the relationship between the coefficients of $R_n$ and $Q_n$ arising from $\sigma(g)$ and $\sigma(gs)$, respectively, when $n \neq i$ and construct automata which recognize this relationship.

\subsection{Analysis of coefficients for type 1 generators.}
\label{sec:type1_noti}
Consider pairs $\otimes(\nu(g),\nu(gs))$ where $s$ is a type 1 generator of the form $ \left( \begin{array}{cc}  t+l_i & b \\ 0 & 1 \end{array} \right)$.  With $g= \g$ we compute
$$gs =  \left( \begin{array}{cc} (t+l_i) \Pi_{k=1}^{d-1} (t+l_k)^{m_k} & b\Pi_{k=1}^{d-1}(t+l_k)^{m_k} + R \\ 0 & 1 \end{array} \right).$$
Using the Decomposition Lemma (Lemma \ref{lemma:decomp}), write
\begin{equation} \label{eqn:R}
\Pi_{k=1}^{d-1} (t+l_i)^{-m_k}R = R_1 + R_2 + \cdots + R_d
\end{equation}
and
\begin{equation}\label{eqn:Q}
\begin{aligned}
(t+l_i)^{-1} \Pi_{k=1}^{d-1} (t+l_k)^{-m_k} (b\Pi_{k=1}^{d-1}(t+l_k)^{m_k} + R) &= b(t+l_i)^{-1} + (t+l_i)^{-1}(R_1 + R_2 + \cdots + R_d) \\ &= Q_1 + Q_2 + \cdots + Q_d
\end{aligned}
\end{equation}
where the latter decomposition into polynomials $Q_k$ is also obtained via the Decomposition Lemma.
We now use Laurent polynomials to describe the exact relationship between $R_n$ and $Q_n$ and show that any differences between them can be detected by a finite state machine.  Namely, we compute the Laurent polynomial ${\mathcal LS}_n$ of the left hand side of Equation \eqref{eqn:Q} and consider the terms of negative degree, which form $Q_n$ for $n \neq d$, and when $n=d$ the terms of nonpositive degree, which form $Q_d$.

First note that ${\mathcal LS}_i(b(t+l_i)^{-1}) = b(t+l_i)^{-1}$ and when $n \neq i$ we see from Equations \eqref{eqn:1a} and \eqref{eqn:2a} that  ${\mathcal LS}_n(b(t+l_i)^{-1})$ contains no terms of negative degree when $n \neq i,d$ and no terms of nonpositive degree when $n=d$.  Thus $b(t+l_i)^{-1}$ will add a term to $Q_i$ but no other $Q_n$ for $n \neq i$.

Assume that $i$ is fixed, since it derives from the generator $s$.
When $k \neq n,d$ we see that when $(t+l_i)^{-1} R_k$ is rewritten as an expression in $t+l_n$ there are no terms of negative degree.  Hence this polynomial contributes no terms to $Q_n$.  To see this, recall that above we wrote
$$R_k = \sum_{z=1}^{\delta_k} \beta_{k,z} (t+l_k)^{-z}$$ where $-\delta_j$ is the minimal degree of $R_k$, and hence a generic term from $(t+l_i)^{-1} R_k$ has the form $\beta (t+l_i)^{-1} (t+l_k)^{-m}$ for some $\beta \in \Z_q$ and $m \in \N$.  We rewrite this in terms of $t+l_n$ using Equations \eqref{eqn:1} and \eqref{eqn:1a} as
$$\beta (t+l_i)^{-1} (t+l_k)^{-m} = \beta \left( \sum_{y=0}^{\infty} \alpha_y (t+l_n)^y \right) \left(\sum_{x=0}^{\infty} \xi_x (t+l_n)^x \right)$$
where the $\xi_x$ and $\alpha_y$ are the coefficients computed in Equations \eqref{eqn:1} and \eqref{eqn:1a}, and note that there are no terms of negative degree.

When $k=d$, the above argument holds with $t+l_k$ replaced by $t^{-1}$ and any application of Equation \eqref{eqn:1} replaced by Equation \eqref{eqn:3}.

Thus it must be the case that the terms of ${\mathcal LS}_n((t+l_i)^{-1} R_n)$ of negative degree form the polynomial $Q_n$.  If $R_n=0$ then $Q_n=0$ as well.  First consider the case $n \neq i,d$. Write
$$R_n = \sum_{k=1}^{\delta_n} \beta_{n,k} (t+l_n)^{-k}$$
and it follows that
$$(t+l_i)^{-1}R_n = \left( \sum_{r=0}^{\infty} \alpha_r (t+l_n)^r \right)  \left(\sum_{k=1}^{\delta_n} \beta_{n,k} (t+l_n)^{-k}\right)$$
where the $\alpha_r$ are computed  in Equation \eqref{eqn:1a}.

To simplify notation in the following argument, write
\begin{equation}\label{eqn:table1}
(t+l_i)^{-1}R_n = \left( \sum_{k \geq 0} \alpha_k (t+l_n)^k \right) \left( \beta_1(t+l_n)^{-1} + \beta_2(t+l_n)^{-2} + \beta_3(t+l_n)^{-3} + \cdots + \beta_{\delta_n}(t+l_n)^{-{\delta_n}} \right)
\end{equation}
where $\beta_i \in \Z_q$ and $\delta_n \neq 0, \ \beta_{\delta_n} \neq 0$.

We see that multiplying $R_n$ by each term in the infinite sum produces a pattern in the resulting coefficients.  As we multiply $R_n$ successively by each term in the infinite sum above, we keep track of the resulting coefficients of the terms of negative degree of $Q_n$ in Table \ref{fig:type1chart}.

\begin{figure}[ht!]
\label{fig:type1chart}
\begin{tabular}{|c|c|c|c|c|c|c|}
\hline
Degree in $Q_n$ & $(t+l_n)^{-1}$ & $(t+l_n)^{-2}$ & $(t+l_n)^{-3}$ &  $\cdots$  & $\beta_{\delta_n-1}(t+l_n)^{-\delta_n+1}$ & $\beta_{\delta_n}(t+l_n)^{-\delta_n}$\\
\hline
Mult $R_n$ by $\alpha_0$ & $\alpha_0 \beta_1$ &  $\alpha_0 \beta_2$ &  $\alpha_0 \beta_3$ &  $\cdots$ & $\alpha_0 \beta_{\delta_n-1}$ & $\alpha_0 \beta_{\delta_n}$ \\
\hline
Mult $R_n$ by $\alpha_1(t+l_n)$ & $\alpha_1 \beta_2$ & $\alpha_1 \beta_3$ & $\alpha_1 \beta_4$ & $\cdots$ & $\alpha_1 \beta_{\delta_n}$ & 0\\
 \hline
Mult $R_n$ by $\alpha_2(t+l_n)^2$ & $\alpha_2 \beta_3$ & $\alpha_2 \beta_4$ &  $\alpha_2 \beta_5$ & $\cdots$ & 0 & 0\\
 \hline
Mult $R_n$ by $\alpha_3(t+l_n)^3$ & $\alpha_3 \beta_4$ & $\alpha_3 \beta_5$ & $\alpha_3 \beta_6$ & $\cdots$ & 0 & 0 \\
 \hline
$\vdots$ & $\vdots$ & $\vdots$ & $\vdots$ & $\vdots$ & $\vdots$ & $\vdots$  \\
\hline
Final Coefficients in $Q_n$ & $\gamma_1$ & $\gamma_2$ & $\gamma_3$ & $\cdots$ & $\gamma_{\delta_n-1}$ & $\gamma_{\delta_n}$ \\
 \hline
\end{tabular}
\caption{As the multiplication in Equation \eqref{eqn:table1} is carried out, like terms are grouped together and we keep track of the resulting coefficients of the terms of negative degree in this table.  Each coefficient $\gamma_k$ of $(t+l_n)^{-k}$ in $Q_n$ is the sum of the coefficients in its column.}
\end{figure}

Now compute the coefficient $\gamma_k$ of  $(t+l_n)^{-k}$ in $Q_n$ by summing the entries of the appropriate column of Table \ref{fig:type1chart}:
\begin{equation} \label{eqn:coef}
\gamma_k = \sum_{x=0}^{\delta_n-k} \alpha_x \beta_{x+k}
\end{equation}
for $1 \leq k \leq \delta_n-1$, and $\gamma_{\delta_n} = C_{n,i}\beta_{\delta_n}$ since $\alpha_0=C_{n,i}$.
Recall that $C_{n,i} = (l_i-l_n)^{-1}$ and hence is invertible.  As $i$ is fixed by our choice of generator, we shorten $C_{n,i}$ to $C_n$ for the remainder of the paper.

Notice that for $k<\delta_n-1$ it follows from Equation \eqref{eqn:coef} that
$$\gamma_k = -C_n \gamma_{k+1} + \alpha_0 \beta_k = -C_n \gamma_{k+1} + C_n \beta_k $$
and hence
$$\gamma_{k+1} = -C_n^{-1}(\gamma_k-C_n \beta_k) = -C_n^{-1} \gamma_k + \beta_k.$$
We will use the expression of $\gamma_{k+1}$ in terms of $\gamma_k$ to construct a finite state machine which recognizes the relationship between the coefficients of $R_n$ and $Q_n$.  It follows from this equation that a pair $(\beta_k,\gamma_k)$ of coefficients of $(t+l_n)^{-k}$ in $R_n$ and $Q_n$, respectively, determine the unique coefficient $\gamma_{k+1}$ of $(t+l_k)^{-(k+1)}$ in $Q_n$.

If we take $n=d$ and replace any instance of Equation \eqref{eqn:1a} with Equation \eqref{eqn:2a}, and Equation \eqref{eqn:1} with Equation \eqref{eqn:2}, we can make an analogous argument. In this case, $R_d$ and $Q_d$ include a constant term. We also rely on the fact that $\alpha_k'$ and $\alpha_{k+1}'$, the coefficients in Equation \eqref{eqn:1a}, are related in the same way as $\alpha_k$ and $\alpha_{k+1}$, the coefficients in Equation \eqref{eqn:1}.  In this argument we use the assumption that the $l_j$ are chosen to be invertible.  We include a few details for completeness.

When $n=d$, write
\begin{equation}\label{eqn:table1d}
(t+l_i)^{-1}R_d = \left( \sum_{k \geq 1} \alpha'_k (t^{-1})^k \right) \left( \beta_0 + \beta_1(t^{-1})^{-1} + \beta_2(t^{-1})^{-2} + \beta_3(t^{-1})^{-3} + \cdots + \beta_{\delta_d}(t^{-1})^{-{\delta_d}} \right)
\end{equation}
where $\beta_i \in \Z_q$ and $\beta_{\delta_d} \neq 0$.  We construct a table analogous to Table \ref{fig:type1chart}, and obtain coefficients $\gamma_0,\gamma_1 \cdots \gamma_{\delta_d}$ for $Q_d$ satisfying
\begin{equation}
\label{eqn:coeffd}
\gamma_k = \sum_{s=1}^{\delta_d-k} \alpha_s' \beta_{s+k}
\end{equation}
for $0 \leq k \leq \delta_d-1$.

As when $n<d$, recall that $\alpha_{s+1}' = -\l_i \alpha_s'$ and these coefficients have the following relationship:
$$\gamma_k = -l_i \gamma_{k+1} + \alpha_1' \beta_{k+1}$$
and hence
$$\gamma_{k+1} = -l_i^{-1} (\gamma_k - \beta_{k+1}).$$
Notice that unlike the case of $n \neq i,d$, the coefficient $\gamma_{k+1}$ depends both on $\gamma_k$ and $\beta_{k+1}$.  This will create different transition functions in the finite state machine constructed below when $n=d$.

A proof of the following lemma is contained in the above exposition when $n \neq i$; it is proven through repeated application of Equations \eqref{eqn:1}, \eqref{eqn:2}, and \eqref{eqn:3} to the expression in Equation \eqref{eqn:Q}.  The case $n=i$ is verified in Section \ref{sec:aut2}.  We state it for easy reference.

\begin{lemma}\label{lemma:lengthsame1}
Fix a type 1 generator $s$, and let $g \in \GG$.  Using the decompositions above in Equations \eqref{eqn:R} and \eqref{eqn:Q}, we see that:
\begin{enumerate}
\item if $n \neq i,d$ then the minimal degree of $R_n$ is the same as the minimal degree of $Q_n$.
\item if $n=d$ then the minimal degree of $Q_i$ is one greater than the minimal degree of $R_i$.
\item if $n=i$ then the minimal degree of $Q_i$ is one less than the minimal degree of $R_i$.
\end{enumerate}
\end{lemma}

\subsection{Analysis of coefficients for type 2 generators.}
\label{subsec:type2}
 We now perform the same analysis on the coefficients of the polynomials which determine $\sigma(g)$ and $\sigma(gs)$ when $s$ is a type 2 generator of the form $\left( \begin{array}{cc} (t+l_i)(t+l_j)^{-1} & b(t+l_j)^{-1} \\ 0 & 1 \end{array} \right)$ with $i<j$.   If $g =\g$, compute
$$gs =  \left( \begin{array}{cc} (t+l_i)(t+l_j)^{-1} \Pi_{m=1}^{d-1} (t+l_i)^{m_i} & b(t+l_j)^{-1}\Pi_{m=1}^{d-1}(t+l_i)^{m_i} + R \\ 0 & 1 \end{array} \right).$$
Apply the Decomposition Lemma to write $\Pi_{m=1}^{d-1} (t+l_i)^{-m_i}R = R_1 + R_2 + \cdots + R_d$ and compute
\begin{equation}
\begin{aligned}(t+l_i)^{-1}(t+l_j) & \Pi_{m=1}^{d-1} (t+l_i)^{-m_i}(b(t+l_j)^{-1}\Pi_{m=1}^{d-1}(t+l_i)^{m_i} + R)  \\  \smallskip
&= b(t+l_i)^{-1} + (t+l_i)^{-1}(t+l_j)R \\ \smallskip
&= b(t+l_i)^{-1} + (t+l_i)^{-1}(t+l_j)(R_1 + R_2 + \cdots + R_d) \\ \smallskip
&= Q_1 + Q_2 + \cdots Q_{d-1} + Q_d
\end{aligned}
\end{equation}
where the last line is obtained by applying the Decomposition Lemma to the original polynomial expression.  We must compute $Q_k$ and show in Section \ref{sec:fsa-nneqi} that its relationship to $R_k$ can be verified by a finite state machine.

When $s$ is a type 2 generator, then the values of $i<j$ and $b$ are fixed: $1 \leq i<j \leq d-1$ and $b \in \Z_q$. Recall that $C_n = (l_i-l_n)^{-1}$ and let $D_n= l_j-l_n$ for the fixed values of $i$ and $j$.

First note that when $n \neq i,d$, we can write $$b(t+l_i)^{-1} = b \sum_{k=0}^{\infty} \alpha_k (t+l_n)^k$$ where $\alpha_k$ is computed in Equation \eqref{eqn:1a}.  As this expression has no terms of negative degree, we do not need to consider $b(t+l_i)^{-1}$ when computing $Q_n$ for $n \neq i,d$.  When $n=d$ we use Equation \eqref{eqn:2a} in place of Equation \eqref{eqn:1a} and observe that there are no terms of nonpositive degree in ${\mathcal LS}_d(b(t+l_i)^{-1})$.  Hence $b(t+l_i)^{-1}$ does not play a role in determining $Q_d$.

Rewrite $(t+l_i)^{-1}(t+l_j)$ as an expression in the variable $t+l_n$ for $n \neq i,j,d$ as
\begin{equation}\label{eqn:Qn}
\begin{aligned}
{\mathcal LS}_n((t+l_i)^{-1}(t+l_j)) &= \left( \sum_{x=0}^{\infty} \alpha_x (t+l_n)^x \right) \left( (l_j-l_n) + (t+l_n) \right) =\left( \sum_{x=0}^{\infty} \alpha_x (t+l_n)^x \right) \left( D_n + (t+l_n) \right)\\
&= C_nD_n + \sum_{x=1}^{\infty} (D_n\alpha_x + \alpha_{x-1}) (t+l_n)^x \\
&= \sum_{x=0}^{\infty} \sigma_x(t+l_n)^x
\end{aligned}
\end{equation}
where $\sigma_0 = C_nD_n$, for $x>1$, we have $\sigma_x = D_n\alpha_x + \alpha_{x-1}$ and $\alpha_x$ is computed as in Equation \eqref{eqn:1a}.  Since $\alpha_{x+1} = -C_n \alpha_x$, it follows that for $x \geq 1$ we have $\sigma_{x+1} = -C_n \sigma_x$.

When $n=j$ the above expression simplifies to
\begin{equation}\label{eqn:n=j}
{\mathcal LS}_j((t+l_i)^{-1}(t+l_j)) = \sum_{x=1}^{\infty} \alpha_{x-1} (t+l_j)^{x}
\end{equation}

When $n=d$ we obtain
\begin{equation}
\label{eqn:n=dtype2}
\begin{aligned}
{\mathcal LS}_d((t+l_i)^{-1}(t+l_j)) &= \left( \sum_{r=1}^{\infty} (-l_i)^{r-1} (t^{-1})^r \right) ((t^{-1})^{-1} + l_j) \\
&= \sum_{r=0}^{\infty} \tau_r (t^{-1})^r
\end{aligned}
\end{equation}
where $\tau_0=1$ and for $x>0$ we have $\tau_x = (-1)^{x-1}(-l_i)^{x-1}(-l_i+l_j)$.

For the remainder of this section, we assume $n \neq i$.
Now we compute ${\mathcal LS}_n((t+l_i)^{-1}(t+l_j)R_k)$ and show that when $k \neq n$ there are no terms of negative (resp. nonpositive when $n=d$) degree, and hence the terms of $Q_n$  are exactly the terms of ${\mathcal LS}_n((t+l_i)^{-1}(t+l_j)R_n)$ of negative (resp. nonpositive) degree.

A generic term in $(t+l_i)^{-1}(t+l_j)R_k$ has the form $(t+l_i)^{-1}(t+l_j) \xi (t+l_k)^{-e}$, with $\xi \in \Z_q$ and $e \in \N$.
When  $n \neq i,j,d$ and $k \neq d$,
\[{\mathcal LS}_n((t+l_i)^{-1}(t+l_j) \xi (t+l_k)^{-e}) =    \left\{
\begin{array}{ll}
 \sum_{x=0}^{\infty} \sigma_x(t+l_n)^x \xi \sum_{z=0}^{\infty} \chi_z (t+l_n)^z & \text{when $ n \neq j,d$} \\
 \\
  \sum_{x=1}^{\infty} \alpha_{x-1}(t+l_j)^x \xi \sum_{z=0}^{\infty} \chi_z (t+l_j)^z & \text{when  $n=j$}
\end{array}
\right. \]
where Equations \eqref{eqn:1}, \eqref{eqn:Qn} and \eqref{eqn:n=j} are used to obtain these expressions,  neither of which contains any terms of negative degree in the variable $t+l_n$.

When $n \neq i,j,d$ and $k=d$ we expand a generic term of $(t+l_i)^{-1}(t+l_j) R_d$ as
\begin{equation}\label{eq:k=d}
{\mathcal LS}_n((t+l_i)^{-1}(t+l_j)\xi (t^{-1})^{-e}) = \sum_{x=0}^{\infty} \sigma_x (t+l_n)^x \xi \sum_{r=0}^{e} {e \choose r} (-l_n)^{e-r} (t+l_n)^r
\end{equation}
which has no terms of negative degree.  In the equation above, $\sigma_x$ is defined in Equation \eqref{eqn:Qn}, $\xi \in \Z_q$ is the coefficient of $(t^{-1})^{-e}$ in $R_d$, and $e \in \N$.
Thus $Q_n$ is comprised of the terms of ${\mathcal LS}_n((t+l_i)^{-1}(t+l_j)R_n)$ of negative degree.

When $n=d$, we use Equations \eqref{eqn:2} and \eqref{eqn:2a} in place of Equations \eqref{eqn:1} and \eqref{eqn:1a}, and Equation \eqref{eqn:n=dtype2} to obtain the expression
$${\mathcal LS}_d((t+l_i)^{-1}(t+l_j) \xi (t+l_k)^{-e}) = \sum_{r=0}^{\infty} \tau_r(t^{-1})^r \xi \sum_{y=e}^{\infty} {e \choose -y} (l_i)^{y+e}(t^{-1})^y$$
and since $e \geq 1$, this has no terms of nonpositive degree.  Thus the terms of $Q_d$ are exactly the terms of ${\mathcal LS}_d((t+l_i)^{-1}(t+l_j)R_d$ of nonpositive degree.

Our computations now mimic those in Section \ref{sec:type1_noti} when $n \neq i$.   When additionally $n \neq j,d$, write ${\mathcal LS}_n((t+l_i)^{-1}(t+l_j)R_n) $ as
$$ \sum_{x=0}^{\infty} \sigma_x(t+l_n)^x  \left( \beta_1(t+l_n)^{-1} + \beta_2(t+l_n)^{-2} + \beta_3(t+l_n)^{-3} +  \cdots + \beta_{\delta_n}(t+l_n)^{-{\delta_n}} \right)$$
 where $\sigma_x$ is defined in Equation \eqref{eqn:Qn}, and create a chart analogous to Figure \ref{fig:type1chart}.  This yields the following formula for the coefficients $\gamma_r$ in $Q_n$:

\begin{equation}\label{eqn:coef2}
\gamma_r = \sum_{k=0}^{\delta_n-r} \sigma_k \beta_{k+r}
\end{equation}

where $\sigma_x$ is defined in Equation \eqref{eqn:Qn} and $\sigma_0=C_nD_n$.  Recall that for $x>1$ we have $\sigma_{x+1} = -C_n \sigma_x$.  Hence, as in the case of type 1 generators, we see that $\gamma_r = -C_n \gamma_{r+1} + \sigma_0 \beta_r = -C_n \gamma_{r+1} + C_nD_n\beta_r$.

When $n=j$, we begin with Equation \eqref{eqn:n=j} and write ${\mathcal LS}_j((t+l_i)^{-1}(t+l_j)R_j$ as
$$\sum_{x=1}^{\infty} \alpha_{x-1}(t+l_j)^x  \left( \beta_{j,1}(t+l_j)^{-1} + \beta_{j,2}(t+l_j)^{-2} + \beta_{j,3}(t+l_j)^{-3} +  \cdots + \beta_{j,\delta_n}(t+l_j)^{-{\delta_j}} \right)$$
from which we construct a chart for the coefficients of $Q_j$ and determine that
$$\gamma_r = \sum_{x=0}^{\delta_j-r} \alpha_x \beta_{x+r+1}.$$
Reasoning analogous to previous cases yields
$$\gamma_{r+1} = -C_j^{-1} (\gamma_r + \beta_{r+1}).$$

When $n=d$, replace instances of Equations \eqref{eqn:1} and \eqref{eqn:1a} with Equations \eqref{eqn:2} and \eqref{eqn:2a} to see that $Q_d$ is the sum of the terms of ${\mathcal LS}_d((t+l_i)^{-1}(t+l_j) R_d)$ of nonpositive degree.  To compute this, we write

\begin{equation}\label{eqn:type2}
\begin{aligned}
(t+l_i)^{-1} &(t+l_j) ( \beta_0 + \beta_1(t^{-1})^{-1} + \beta_2(t^{-1})^{-2} + \cdots + \beta_{\delta_d}(t^{-1})^{-\delta_d}) \\
&= \sum_{x=1}^{\infty} (-1)^{x-1} l_i^{x-1} (t^{-1})^x \cdot ((t^{-1})^{-1} + l_j) ( \beta_0 + \beta_1(t^{-1})^{-1} + \beta_2(t^{-1})^{-2} + \cdots + \beta_{\delta_d}(t^{-1})^{-\delta_d}) \\
&= \left( \sum_{x=0}^{\infty} (-1)^{x} l_i^{x} (t^{-1})^x + \sum_{x=1}^{\infty} (-1)^{x-1} l_i^{x-1} l_j(t^{-1})^x \right)( \beta_0 + \beta_1(t^{-1})^{-1} + \cdots + \beta_{\delta_d}(t^{-1})^{-\delta_d}) \\
&= \left(\sum_{x=0}^{\infty} \sigma_x' (t^{-1})^x\right)( \beta_0 + \beta_1(t^{-1})^{-1} + \beta_2(t^{-1})^{-2} + \cdots + \beta_{\delta_d}(t^{-1})^{-\delta_d})
\end{aligned}
\end{equation}

where $\sigma_0'=1$ and $\sigma_x'= (-1)^x l_i^x + (-1)^{x-1} l_i^{x-1} l_j$, from which it follows that $\sigma_{x+1}'=(-l_i) \sigma_x$.
Hence when we compute the coefficients $\gamma_r$ of $Q_d$ we see that the relationship is identical to the case of type 1 generators.  The difference between type 1 and type 2 generators when $n=d$ is that the minimal degree of the expression in Equation \eqref{eqn:type2} is the same as the minimal degree of $R_d$.  This is because the expression for $(t+l_i)^{-1}(t+l_j)$ written in terms of $t^{-1}$ has a constant term, which is not the case when $(t+l_i)^{-1}$ is written in terms of $t^{-1}$.  Explicitly, we compute that
$$\gamma_k = \sum_{v=0}^{\delta_d-k} \sigma_v'\beta_{v+k}$$
which is identical to the expression in Equation \eqref{eqn:coeffd} for type 1 generators except that the index begins at $0$.  Computations then yield
$$\gamma_k = -l_i \gamma_{k+1} + \sigma_0' \beta_k$$ and hence $$\gamma_{k+1} = -l_i^{-1} (\gamma_k- \beta_k).$$

We now state the following lemma, whose proof for $n \neq i$ is contained in the verification of the above expressions.  The case $n=i$ is verified in Section \ref{sec:aut2}.
\begin{lemma}\label{lemma:lengthsame2}
Fix a type 2 generator $s$, and let $g \in \GG$.  Using the decompositions above in Equations \eqref{eqn:R} and \eqref{eqn:Q}, we see that:
\begin{enumerate}
\item if $n \neq i,j$ then the minimal degree of $R_n$ is the same as the minimal degree of $Q_n$.
\item if $n=i$ then the minimal degree of $Q_i$ is one less than the minimal degree of $R_i$.
\item if $n=j$ then the minimal degree of $Q_j$ is one greater than the minimal degree of $R_j$.
\end{enumerate}
\end{lemma}

\subsection{Construction of automata when $n \neq i$.}
\label{sec:fsa-nneqi}

We now construct, for each $1 \leq n \leq d, \ n \neq i$, and $\epsilon \in \{1,2\}$ a finite state machine $\overline{\MM}$ which accepts the convolution $\otimes(\sigma(R_n),\sigma(Q_n))$.   The value of $\epsilon$ indicates whether  $s$ is a type 1 or type 2 generator.  That is, the machine accepts strings of the form
\begin{equation} \label{eqn:nneqd}
{\beta_1 \choose \gamma_1} {\beta_2 \choose \gamma_2}{\beta_3 \choose \gamma_3}  \cdots {\beta_{\delta_n-1} \choose \gamma_{\delta_{n}-1}} {\beta_{\delta_n} \choose \gamma_{\delta_n}}
\end{equation}
except when ($\epsilon=1$ and $n=d$) or ($\epsilon=2$ and $n=j$) and in those two cases, strings of the form
\begin{equation}\label{eqn:n=d}
{\beta_1 \choose \gamma_1} {\beta_2 \choose \gamma_2}{\beta_3 \choose \gamma_3}  \cdots {\beta_{\delta_n-1} \choose \gamma_{\delta_{n}-1}} {\beta_{\delta_n} \choose \diamond}
\end{equation}
where the relationship between the $\beta_x$ and $\gamma_x$ is explicitly described in the previous two sections.   If $R_n=0$ it follows that $Q_n=0$ as well, and we adapt the machines to accept this pair as well.

For each $n \neq i$ construct a finite state machine $\overline{\MM}$ as follows.
\begin{enumerate}
\item  Create states $T_{\sigma,\tau}$ for all $\sigma,\tau \in \Z_q$, and for a given pair $(\sigma, \tau)$:
    \begin{enumerate}
    \item when $n \neq d$ and

    \smallskip

    \begin{enumerate}
    \item $\epsilon=1$, compute the least residue $\gamma$ of $ -C_n^{-1}\tau + \sigma (mod \ q)$.

    \smallskip

    \item $\epsilon=2$ and $n \neq j$, compute the least residue $\gamma$ of $ -C_n^{-1}\tau + D_n \sigma (mod \ q)$.
    \end{enumerate}

    \smallskip

    For each $\beta \in \Z_q$, add a transition arrow from $T_{\sigma, \tau}$ with label ${\beta \choose \gamma}$ to state $T_{\beta , \gamma}$.

    \medskip

    \item when $\epsilon=2$ and $n=j$, for each $\beta \in \Z_q$, compute the least residue $\gamma$ of $ -C_n^{-1}\tau + \beta$ and add a transition arrow from $T_{\sigma, \tau}$ with label ${\beta \choose \gamma}$ to state $T_{\beta , \gamma}$.

    \medskip

    \item when $\epsilon=1$, and $n=d$, compute the least residue $\gamma$ of $-l_i^{-1}(\tau-\beta) (mod \ q)$  for each $\beta \in \Z_q$.  Add a transition arrow from $T_{\sigma,\tau}$ to $T_{\beta,\gamma}$ with label ${\beta \choose \gamma}$.

    \medskip

    \item when $\epsilon=2$ and $n=d$, compute the least residue $\gamma$ of $-l_i^{-1}(\tau-\sigma) (mod \ q)$.  For each $\beta \in \Z_q$, add a transition arrow from $T_{\sigma, \tau}$ with label ${\beta \choose \gamma}$ to state $T_{\beta , \gamma}$.
    \end{enumerate}

\medskip

\item Add a start state with $q^2$ transition edges, with labels ${\beta \choose \gamma}$ for all $\beta,\gamma \in \Z_q$.  The edge with label ${\beta \choose \gamma}$ terminates at $T_{\beta,\gamma}$.  Allow the start state to be an accept state so that the empty pair $R_n=Q_n=0$ is accepted.

\medskip

\item Introduce an accept state $A$; from each state $T_{\sigma,\tau}$:
 \begin{enumerate}
 \item when $n \neq d$ and

    \smallskip

 \begin{enumerate}
    \item $\epsilon=1$, compute the least residue $\gamma$ of $ -C_n^{-1}\tau + \sigma (mod \ q)$ and the least residue $\beta$ of $C_n^{-1} \gamma (mod \ q)$.

    \smallskip

    \item $\epsilon=2$ and $n \neq j$, compute the least residue $\gamma$ of $ -C_n^{-1}\tau + D_n \sigma (mod \ q)$ and the least residue $\beta$ of $(C_nD_n)^{-1} \gamma (mod \ q)$.
    \end{enumerate}
    Add a single transition from this state to $A$ with label ${\beta \choose \gamma}$.

    \medskip

 \item when $\epsilon=2$ and $n = j$, compute the least residue $\beta$ of $-C_j^{-1} \tau$, and add a single transition to state $A$ with label ${\beta \choose \diamond}$.

     \medskip

 \item when $n=d$, and

    \smallskip

 \begin{enumerate}
    \item $\epsilon=1$, let $\beta=\tau$ and add a single transition to state $A$ with label ${\beta \choose \diamond}$.

    \smallskip

    \item $\epsilon=2$, compute the least residue $\gamma$ of $-l_i^{-1}(\tau-\sigma) (mod \ q)$.  Add a single transition to state $A$ with label ${\gamma \choose \gamma}$.
 \end{enumerate}
\end{enumerate}
\end{enumerate}

Any word accepted by $\overline{\MM}$ corresponds to two nonempty strings $\beta_1 \beta_2 \cdots \beta_k$ and $\gamma_1 \gamma_2 \cdots \gamma_{\eta}$ (for $\eta=k$ or $k-1$) where the coefficients differ according to Equation \eqref{eqn:coef} or \eqref{eqn:coef2}.
Thus $\overline{\MM}$ accepts the language of convolutions of the form $\otimes(\sigma(R_n),\sigma(Q_n))$, $n \neq i$, where $R_n$ and $Q_n$ arise from $g$ and $gs$, respectively.

We now construct a machine $M_{\epsilon}$ which accepts $\otimes(\sigma(g),\sigma(h))$ if for each $1 \leq n \leq d$, $n \neq i$, the $n$-th substring in the convolution $\otimes(\sigma(g),\sigma(h))$ is related in the manner proscribed by $\overline{\MM}$.  Recall from Lemmas \ref{lemma:lengthsame1} and \ref{lemma:lengthsame2} that for
\begin{itemize}
\item $n \neq i,j,d$, we have $|\sigma(R_n)| = |\sigma(Q_n)|$,
\smallskip
\item $n=i$, we have $|\sigma(R_i)|+1 = |\sigma(Q_i)|$,
\smallskip
\item $n=j$, if $\epsilon=1$ then  $|\sigma(R_n)| = |\sigma(Q_n)|$ and if $\epsilon = 2$ then $|\sigma(R_d)|-1 = |\sigma(Q_d)|$,
and
\smallskip
\item $n=d$, if $\epsilon=1$ then $|\sigma(R_d)|-1 = |\sigma(Q_d)|$, and if $\epsilon=2$ then $|\sigma(R_d)| = |\sigma(Q_d)|$.
\end{itemize}
This list describes the offset in $\otimes(\sigma(g),\sigma(h))$ between the strings $\sigma(R_n)$ on the top line of the convolution and $\sigma(Q_n)$ on the bottom line of the convolution  when a finite state machine reads $\otimes(\sigma(g),\sigma(gs))$.  If $\epsilon=1$ then for $n>i$ the strings $\sigma(R_n)$ and $\sigma(Q_n)$ are offset by $1$ as the convolution $\otimes(\sigma(g),\sigma(gs))$ is read.  If $\epsilon=2$ then for $i<k \leq j$ the strings $\sigma(R_k)$ and $\sigma(Q_k)$ are offset by $1$, and for $k>j$ they are aligned, and hence read simultaneously.

Let $\Phi_n$ and $\Psi_n$ be (possibly empty) strings of length $\eta_n \geq 0$ and $\chi_n \geq 0$ of symbols from the finite alphabet consisting of elements of $\Z_q$, for $1 \leq n \leq d$.  Let ${\mathcal K}_{\epsilon}$ be the language of convolutions of the form $\otimes(\Phi_1\#\Phi_2\#\cdots\# \Phi_d,\Psi_1\#\Psi_2\#\cdots \#\Psi_d)$ where for $n \neq i$ we assume without loss of generality (as these conditions can be verified with finite state automata) that the lengths $\eta_k$ and $\chi_k$ agree with Lemma \ref{lemma:lengthsame1} if $\epsilon=1$ and  Lemma \ref{lemma:lengthsame2} if $\epsilon=2$.   We view the two strings in the convolution as arising from $\sigma(g)$ and $\sigma(h)$ for some $g,h \in \GG$.  It is clear that ${\mathcal K}_{\epsilon}$ is a regular language.  We want to show that the subset ${\mathcal K}'_{\epsilon}$ of ${\mathcal K}_{\epsilon}$ in which $\otimes(\Phi_n,\Psi_n)$ is accepted by $\overline{\MM}$, for all $1 \leq n \leq d$, $n \neq i$ is also a regular language.

Let $\overline{{\mathcal K}_{\epsilon}}$ consist of strings of the same form as those in ${\mathcal K}_{\epsilon}$ with the condition that $\eta_i = \chi_i$, that is, the strings $\Phi_i$ and $\Psi_i$ have the same length. We impose no other conditions on the remaining strings $\Phi_n$ and $\Psi_n$.   When $\epsilon=1$, this language is accepted by the machine ${\mathcal M'}_{1}$ constructed  as follows.
\begin{enumerate}
\item The start state of  ${\mathcal M'}_1$  is the start state of the machine $\overline{M_{1,1}}$.
\item For $1 \leq n \leq i-2$, add a transition arrow with label ${\# \choose \#}$ between the accept state $A$ of $\overline{M_{n,1}}$ and the start state of $\overline{M_{n+1,1}}$.
\item Add a transition arrow with label ${\# \choose \#}$ from the start state of $\overline{M_{n-1,1}}$ to the start state of $\overline{M_{n,1}}$, for all $2 \leq n \leq d$ with $n \neq i$.
\item Add a transition arrow with label ${\# \choose \#}$ from the accept state $A$ of $\overline{M_{i-1,1}}$ to a state $S_i$.  Add a loop at $S_i$ with label ${\beta \choose \gamma}$ for each $\beta,\gamma \in \Z_q$.  From $S_i$ add a transition arrow to the start state of $\overline{M_{i+1,1}}$ with label${\# \choose \#}$.
\item For $i+1 \leq n \leq d-1$, add a transition arrow with label ${\# \choose \#}$ between the accept state $A$ of $\overline{M_{n,1}}$ and the start state of $\overline{M_{n+1,1}}$.
\item Let the accept state $A$ of $\overline{M_{d,1}}$ be the accept state of the entire machine.
\end{enumerate}
It then follows from  Lemma \ref{lemma:offset} that ${\mathcal K}'_1$ is a regular language.  This is exactly the language of convolutions $\otimes(\sigma(g),\sigma(h))$ where all but the $i$-th substrings have the same relationship as if $h=gs$ where $s$ is a type 1 generator.

When $\epsilon=2$ we assume that in $\overline{{\mathcal K}_{2}}$ the strings $\Phi_i$ and $\Psi_i$ have the same length, as do the strings $\Phi_j$ and $\Psi_j$.  We construct ${\mathcal M}_2$ as above, with one modification.  Introduce a state $S_j$ which replaces $\overline{M_{j,2}}$ analogous to $S_i$ above.  The resulting machine accepts convolutions where all but the $i$-th and $j$-th substrings differ in the proscribed manner for a type 2 generator.
The language accepted by this machine is regular, and it follows from Lemma  \ref{lemma:offset} that ${\mathcal K}'_2$, in which $|\Psi_i|=|\Phi_i|+1$ and $|\Psi_j|=|\Phi_j|-1$, is a regular language.  Create a simple finite state machine which additionally verifies that $\Phi_j$ and $\Psi_j$ have the relationship given by $\overline{M_{j,2}}$.  We then conclude that the language of convolutions ${\mathcal K}_2$ of the form $\otimes(\sigma(g),\sigma(h))$ where all but the $i$-th substrings have the same relationship as if $h=gs$ where $s$ is a type 2 generator, is a regular language.

\section{Construction of Automata II}
\label{sec:aut2}

We now determine the relationship between the coefficients of $R_i$ and $Q_i$ arising from $\sigma(g)$ and $\sigma(gs)$, where $i$ is fixed by the choice of generator $s$. The coefficient of $(t+l_i)^{-1}$ in $Q_i$ is given by a complicated sum, and we construct a separate automaton simply to that this coefficient is correct. A second automaton is constructed to verify the relationship between the remaining coefficients of $R_i$ and $Q_i$.

\subsection{Analysis of coefficients for type 1 generators.}
\label{sec:type1}
Suppose that $\sigma(g)$ and $\sigma(gs)$ are as above, where $s$ is a type 1 generator.  First we compute the coefficient of $(t+l_i)^{-1}$ in $Q_i$, beginning with the expression in Equation \eqref{eqn:Q}.  We compute this coefficient as a running sum, which we refer to as a partial sum $\sigma$, as each term in each $R_n$ for $n \neq i$ will contribute a term to the final sum which becomes the coefficient of $(t+l_i)^{-1}$ in $Q_i$.

First note that the $b(t+l_i)^{-1}$ term in Equation \eqref{eqn:Q} will contribute $b$ to the partial sum $\sigma$ which will become the coefficient of the $(t+l_i)^{-1}$ term in $Q_i$.  Additionally, $(t+l_i)^{-1}R_i$ only contains terms of degree at most $-2$ and hence does not contribute any terms to $\sigma$.

We now compute the contribution to $\sigma$ from a generic term in $(t+l_i)^{-1}R_n$ for $n \neq i$; when $n \neq d$ as well, such a term has the form $(t+l_i)^{-1}\xi(t+l_n)^{-e}$ where $\xi \in \Z_q$ and $e \in \N$. Using Equation \eqref{eqn:1} or Equation \eqref{eqn:3} (when $n=d$) we see that
$$(t+l_i)^{-1}\xi(t+l_n)^{-e} =(t+l_i)^{-1} \xi \sum_{k=0}^{\infty} \chi_k (t+l_i)^k $$
which has a single term of negative degree, namely $\xi \chi_0 (t+l_i)^{-1}$.  Here, $\chi_k$ is the coefficient computed in Equation \eqref{eqn:1} and $\chi_0=C_n^e$.  When $n=d$ we obtain an analogous equation with $\chi_0$ computed as in Equation \eqref{eqn:3}, in which case $\chi_0=(-l_i)^e$.   So each term of $R_n$ contributes its constant term (when expanded in the variable $t+l_i$) to the sum $\sigma$.
As above, write
\begin{align*}
R_n &= \sum_{k=1}^{\delta_n} \beta_{n,k} (t+l_n)^{-k} \\
&= \sum_{k=1}^{\delta_n} \beta_{n,k} \sum_{r=0}^{\infty} {k \choose r} C_n^{r-k} (t+l_i)^r
\end{align*}
when $n \neq d$, and when $n=d$,
$$
R_d = \sum_{k=1}^{\delta_d} \beta_{d,k} \sum_{r=0}^{k} {k \choose r} (-l_i)^{k-r} (t+l_i)^r.
$$

As we are only interested in the terms in the above sum when $r=0$, we see that the coefficient of $(t+l_i)^{-1}$ in $Q_i$ must be
\begin{equation}\label{eqn:firstcoeff}
b+ \sum_{t=1, t \neq i}^{d-1} \sum_{j=1}^{\delta_t} \beta_{t,j} C_t^j + \sum_{y=0}^{\delta_d} \beta_{d,y} (-l_i)^y
\end{equation}
where the values of $b$ and $i$ are determined by the original generator $s$.

Next, we ignore the coefficient of $(t+l_i)^{-1}$ and then the above reasoning shows that the terms of $Q_i$ of degree at most $-2$ are given by the terms of ${\mathcal LS}_i((t+l_i)^{-1} R_i)$.  Write
$$(t+l_i)^{-1} R_i = (t+l_i)^{-1}  \sum_{x=1}^{\delta_i} \beta_{i,x} (t+l_i)^{-x} = \sum_{x=1}^{\delta_i} \beta_{i,x} (t+l_i)^{-(x+1)}$$
and notice that the coefficients of $R_i$ are shifted over to form the coefficients of $Q_i$ of degree at most $-2$.  The minimal degree of $Q_i$ is $-(\delta_i+1)$ with coefficient $\beta_{i,\delta_i} \neq 0$.  Thus we are comparing strings of coefficients of the form
\begin{equation}\label{eqn:remainingcoeff}
\begin{array}{ccccccc} \beta_{i,1} & \beta_{i,2}  &  \beta_{i,3}  &  \cdots  & \beta_{i,\delta_i-1}  & \beta_{i,\delta_i}  & \# \\
\xi  & \beta_{i,1}  & \beta_{i,2}  &  \cdots  &  \beta_{i,\delta_i-2}   & \beta_{i,\delta_i-1} &  \beta_{i,\delta_i}
\end{array}
\end{equation}
where $\xi \in \Z_q$ can be any element, since we are not concerned in this step with the relationship between the coefficient of $(t+l_i)^{-1}$ in $R_i$ and $Q_i$.

\subsection{Construction of finite state machines for type 1 generators when $n=i$.}
\label{sec:fsa-type1-n=i}
Consider again the language ${\mathcal K}_1$ defined in Section \ref{sec:aut1}; this is the language of all possible strings $\otimes(\sigma(g),\sigma(h))$ for $g,h \in \GG$.  These strings have the form $\Phi_1 \# \Phi_2 \# \cdots \# \Phi_d$ and $\Psi_1 \# \Psi_2 \# \cdots \# \Psi_d$, respectively, where each $\Phi_k$ and $\Psi_k$ is a string of elements of $\Z_q$, and we assume without loss of generality that the lengths of the corresponding substrings are related as in Lemma \ref{lemma:lengthsame1}.

In this section, we must show that the subset ${\mathcal H}_1$ of ${\mathcal K}_1$ in which the first entry of $\Psi_i$ relates to the entire string $\Phi_1 \# \Phi_2 \# \cdots \# \Phi_d$ as specified in Equation \eqref{eqn:firstcoeff}, and then the subset ${\mathcal H}_2$ of strings where the remaining entries of $\Phi_i$ and $\Psi_i$ differ as in Equation \eqref{eqn:remainingcoeff}, are regular languages.  Their intersection is then a regular language ${\mathcal H}$ in which $\Phi_i$ and $\Psi_i$ differ as in $\otimes(\sigma(g),\sigma(gs))$ where $s$ is a type 1 generator.

We first construct a machine $M_{{\mathcal H},1}$ which accepts exactly the set ${\mathcal H}_1$. This machine stores a partial sum which is augmented as each pair ${\beta \choose \gamma}$ is read from $\otimes(\sigma(g),\sigma(h)) \in {\mathcal K}_1$, although only the value of $\beta$ increases the sum.  To accept a string, this value is compared against the first entry in $\Psi_i \subset \sigma(h)$.   As the machine has no memory, these values are stored implicitly in the indexing of the states and the transition functions.

For each value of $n$ with $1 \leq n \leq d$, $n \neq i$, consider the cycle ${\mathcal C}_n = \{C_n,C_n^2,C_n^3 \cdots ,C_n^{k_n}=1\}$ of length $k_n$ where all values are taken $mod \ q$, and $C_n$ is defined in Section \ref{sec:nf}. Create a set of $q^2 k_n$ states of the form $T_{\alpha,e,\sigma}$ where $\alpha, \sigma \in \Z_q$, $e \in \{1,2,3, \cdots,k_n\}$, where
\begin{itemize}
\item $\alpha$ stores the coefficient of the term of $R_n$ that we are reading (denoted $\beta_{n,j}$ above),
\item $e$ is the exponent of $C_n$ in the cycle ${\mathcal C}_n$, and
\item $\sigma$ is the partial sum of the coefficient of $(t+l_i)^{-1}$ in $Q_i$.
\end{itemize}

From state $T_{\alpha,e,\sigma}$, for each $\beta \in \Z_q$ and $\gamma \in \Z_q \cup \{\#\}$, compute $\sigma'$ to be the least residue of $\sigma+\beta C_n^{e+1}(mod \ q)$ (resp. $\sigma+ \beta (-l_i)^{e+1}$ when $n=d$) and introduce a transition labeled ${\beta \choose \gamma}$ to $T_{\beta, e+1,\sigma'}$.  To streamline notation, we always assume that the second coordinate in the state index is reduced $mod \ k_n$, and the third index is reduced $mod \ q$.
Denote the resulting machine $N_n$; note that this step creates $d-1$ distinct finite state automata.  Note that we have not added any start or accept states to this machine yet.

To create the composite machine $M_{{\mathcal H},1}$, begin with a start state with $q^2$ transition arrows emanating from it, with labels ${\beta \choose \gamma}$ for all $\beta,\gamma \in \Z_q$.  The arrow with label ${\beta \choose \gamma}$ terminates at state $T_{\beta,1,b+\beta C_1}$ in $N_1$.

To transition from $N_c$ to $N_{c+1}$ for $c \leq i-2$, create a set of $q$ states $S_{c,0},S_{c,1},S_{c,2}, \cdots S_{c,q-1}$ reflecting in the second coordinate the possible values of the partial sum $\sigma$. From each state $T_{\alpha,e,\sigma}$ of $N_{c}$ introduce a transition with label ${\# \choose \#}$ to the state $S_{c,\sigma}$.  From each state $S_{c,\sigma}$ introduce $q^2$ transitions, where the transition with label ${\beta \choose \gamma}$ terminates at the state $S_{\beta,1,\sigma+\beta C_{c+1}^1}$ of $N_{c+1}$, for each pair $\beta,\gamma \in \Z_q$.  As we pass from $N_c$ to $N_{c+1}$ in this way we are transitioning from reading the coefficients of $R_c$ to the coefficients of $R_{c+1}$ and the information we must retain in terms of the state indexing is the partial sum $\sigma$.

Now add one additional arrow from the start state with label ${\# \choose \#}$ which terminates at state $S_{1,b}$ where $b$ is fixed in the generator $s$.  From each state $S_{c,p}$ add a transition with label ${\# \choose \#}$ which terminates at state $S_{c+1,p}$.  This corresponds to $\Phi_c = \emptyset$ for $c \leq i-2$.

Now we have ``connected" the machines $N_1$ through $N_{i-1}$.  As above, create $q$ states labeled $S_{i-1,0},S_{i-1,1},S_{i-1,2}, \cdots S_{i-1,q-1}$ reflecting the possible values of $\sigma$ in the second coordinate. From each state $T_{\alpha,e,\sigma}$ of $N_{i-1}$ introduce a transition with label ${\# \choose \#}$ to the state $S_{i-1,\sigma}$.  Create $q^2$ states $S_{i,a,b}$ for each pair $a,b \in \Z_q$; the index $a$ stores the value of $\sigma$ and the index $b$ is the coefficient of $(t+l_i)^{-1}$ in $Q_i$ if we began with a pair $g,gs$, otherwise it is the first entry in $\Psi_i$.  We denote this entry by $\psi_{i,1}$.  From state $S_{i-1,c}$ for each $\beta \in \Z_q$ add a transition with label ${\beta \choose \gamma}$ to state $S_{i,c,\gamma}$.  At each state $S_{i,a,b}$ add a loop with label ${a \choose b}$ for $a \in \Z_q \cup \{\#\}$ and $b \in \Z_q$.

Next, create $q$ copies of $N_{i+1}$, which we denote $N_{i+1,p}$ for $p \in \Z_q$.  No transitions will be introduced between copies of $N_{i+1,p}$ and $N_{i+1,p'}$ for $p \neq p'$.  From state $S_{i,\sigma,\psi_{i,1}}$, add a transition arrow with label ${\beta \choose \#}$ to the state $T_{\beta,1,\sigma+\beta C_i}$ of $N_{i+1,\psi_{i,1}}$. The lack of transitions between the $N_{i,p}$ retains the value of $\psi_{i,1}$.

Create $q$ copies of $N_r$ for $i+2 \leq r \leq d$ which we index by $N_{r,p}$ for $p \in \Z_q$, and corresponding transition states $S_{r,p}$ as above.  Mimic the transitions between machines as above, with two changes.  To connect the machines $N_{r,p}$ and $N_{r+1,p}$ via the intermediate states $S_{k,p}$:
\begin{itemize}
\item replace any transition with label ${\# \choose \#}$  from $N_{r,p}$ to $S_{k,p}$ by a set of transitions with labels ${\# \choose \xi}$, for $\xi \in \Z_q$, and
\item replace a transition with label ${a \choose b}$ with $a,b \in \Z_q$ from $S_{k,p}$ to $N_{r+1,p}$ by  set of transitions with labels ${\chi \choose \#}$, for $\chi \in \Z_q$.
\end{itemize}

To finish the construction of the machine $M_{{\mathcal H},1}$ which accepts the language ${\mathcal H}_1$, we must verify that the final sum $\sigma$ is exactly the coefficient $\psi_{i,1}$.  To accomplish this, in $N_{d,p}$ designate only $T_{\alpha,e,p}$ as an accept state.

It follows directly from Lemma \ref{lemma:shift} that the language of convolutions of strings $\otimes(\sigma(g),\sigma(h))$ for which all but the initial coefficients in $\Phi_i$ and $\Psi_i$ are related as in Equation \eqref{eqn:remainingcoeff} form a regular language ${\mathcal H}_2$.
Let ${\mathcal H} = {\mathcal H}_1 \cap {\mathcal H}_2$; then ${\mathcal H}$ is the language of those convolutions where $\Phi_i$ and $\Psi_i$ are related as in $\otimes(\sigma(g),\sigma(gs))$ where $s$ is a type 1 generator. Hence ${\mathcal N}_s' = {\mathcal K}_1 \cap {\mathcal H}$ is a regular language which contains all convolutions of the form $\otimes(\sigma(g),\sigma(gs))$ where $s$ is a type 1 generator.

\subsection{Analysis of coefficients for type 2 generators.}
We now mimic the analysis of the coefficients of $\sigma(R_i)$ and $\sigma(Q_i)$ where $R_i$ and $Q_i$ arise from $\sigma(g)$ and $\sigma(gs)$ and $s$ is a type 2 generator.  Recall that we must convert
$$b(t+l_i)^{-1} + (t+l_i)^{-1}(t+l_j)(R_1 + R_2 + \cdots + R_d)$$
to a Laurent polynomial in the variable $t+l_i$.  As before, the coefficient of $(t+l_i)^{-1}$ in $Q_i$ will be computed as a running sum $\sigma$, of which $b$ will be a summand.

Note that $(t+l_i)^{-1}(t+l_j)=1+(l_j-l_i)(t+l_i)^{-1}$ and the initial $1$ creates the difference between the case when $s$ is a type 1 generator, and this case.  In particular, when we compute $(t+l_i)^{-1}(t+l_j)R_i$ we see that
\begin{equation}\label{eqn:type2i}
\left(1+(l_j-l_i)(t+l_i)^{-1}\right)\sum_{x=1}^{\delta_i} \beta_{i,x} (t+l_i)^{-x} = \sum_{x=1}^{\delta_{i}+1} \tau_x (t+l_i)^{-x}
\end{equation}
where
\begin{itemize}
\item $\tau_1=\beta_{i,1}$,
\item $\tau_k=\beta_{i,k}+(l_j-l_i)\beta_{i,k-1}$ for $2 \leq k \leq \delta_i$, and
\item  $\tau_{\delta_i+1}=(l_j-l_i)\beta_{i,\delta_i}$.
\end{itemize}
We notice immediately that this Laurent polynomial contributes its initial coefficient, $\beta_{i,1}$ to the coefficient $\sigma$ of $(t+l_i)^{-1}$ in $Q_i$, which is not the case when $s$ is a type 1 generator.

We now compute the contribution to $\sigma$ from a generic term in $(1+(l_j-l_i)(t+l_i)^{-1})R_n$ for $n \neq i,d$, which is of the form $(1+(l_j-l_i)(t+l_i)^{-1}) \xi (t+l_n)^{-e}$, for $\xi \in \Z_q$.  Using Equation \eqref{eqn:1} we see that
\begin{align*}
\left(1+(l_j-l_i)(t+l_i)^{-1}\right) \xi (t+l_n)^{-e} &= (1+(l_j-l_i)(t+l_i)^{-1})\xi \sum_{r=0}^{\infty} {-e \choose r} (l_i-l_n)^{-e-r} (t+l_i)^r \\
&= (1+(l_j-l_i)(t+l_i)^{-1})\xi \sum_{r=0}^{\infty} \chi_r (t+l_i)^r
\end{align*}
in which the coefficient of $(t+l_i)^{-1}$ is $(l_j-l_i)\xi \chi_0 = (l_j-l_i) \xi C_n^{e}$.  As this differs from the case when $s$ is a type 1 generator only by a constant, namely $l_j-l_i$, we  see that the identical analysis applies to computing the contribution to $\sigma$ from $(1+(l_j-l_i)(t+l_i)^{-1})R_n$ when $n \neq i,d$.   When $n=d$, we use Equation \eqref{eqn:3} instead of Equation \eqref{eqn:1} to obtain
$$(1+(l_j-l_i)(t+l_i)^{-1}) \xi (t^{-1})^{-e} = (1+(l_j-l_i)(t+l_i)^{-1}) \xi \sum_{r=0}^e {e \choose r} (-l_i)^{e-r} (t+l_i)^n.$$
from this we see that each term of this form contributes $(l_i-l_j) \xi (-l_i)^e$ to the sum $\sigma$, and the constant term $\beta_{d,0}$ contributes $(l_j-l_i) \beta_{d,0}$ to this sum.

Thus the coefficient of $(t+l_i)^{-1}$ in $Q_i$ is
\begin{equation}\label{eqn:firstcoefftype2}
b+ \beta_{i,1} + (l_j-l_i) \left( \sum_{t=1, t \neq i}^{d-1} \sum_{j=1}^{\delta_t} \beta_{t,j} C_t^j + \sum_{j=0}^{\delta_d} \beta_{d,j} (-l_i)^j  \right)
\end{equation}
where the values of $i,j$ and $b$ are determined by the original generator $s$.  This expression is very close to the one in Equation \eqref{eqn:firstcoeff}: there are two initial terms instead of one, and the summation is multiplied by a constant.  The automaton constructed in Section \ref{sec:fsa-type1-n=i} is easily adapted to account for these minor changes in the sum, and we obtain the same conclusions as in Section \ref{sec:type1} but for type 2 generators.

It now follows that the terms of $Q_i$ of degree at most $-2$ are given by the terms of
$$(t+l_i)^{-1}(t+l_j)R_i = (1+(l_j-l_i)(t+l_i)^{-1})R_i =  \sum_{x=1}^{\delta_{i}+1} \tau_x (t+l_i)^{-x}$$
for $\tau_x$ defined in Equation \eqref{eqn:type2i}.  Letting $D_i=l_j-l_i$ we create a simple automaton which accepts strings of the form:
\begin{equation}\label{eqn:coeff}
{ \beta_{i,1} \choose \beta_{i,1}} { \beta_{i,2} \choose D_i \beta_{i,1} + \beta_{i,2}}{ \beta_{i,3} \choose D_i \beta_{i,2} + \beta_{i,3}} \cdots { \beta_{i,\delta_i} \choose D_i \beta_{i,\delta_i-1} + \beta_{i,\delta_i}}{ \# \choose D_i \beta_{i,\delta_i} }
\end{equation}
where all coordinates are commputed modulo $q$.  Create $q^2$ states $T_{a,b}$ for each $a,b \in \Z_q$.  From a start state $S$, add transition arrows with label ${a \choose a}$ terminating at state $T_{a,a}$.  From state $T_{a,b}$, compute $d$ to be the least residue mod $q$ of $aD_i+c$, for each $c \in \Z_q$ and add a transition with label ${c \choose d}$ which terminates at state $T_{c,d}$.  From each state $T_{a,b}$ add a transition with label ${\# \choose aD_i}$ to an accept state.  Then this machine can be easily extended to a machine which verifies that in strings $\Phi=\Phi_1 \# \Phi_2 \# \cdots \# \Phi_d$ and $\Psi=\Psi_1 \# \Psi_2 \# \cdots \# \Psi_d$, all but the initial coefficients of $\Phi_i$ and $\Psi_i$ differ as in expression \eqref{eqn:coeff}.

Thus we have shown that the set of all convolutions $\otimes(\sigma(g),\sigma(h))$ for which the initial coefficients in $\Phi_i$ and $\Psi_i$ are related as in Equation \eqref{eqn:firstcoefftype2} form a regular language, as do the set of convolutions $\otimes(\sigma(g),\sigma(h))$ for which all but the initial coefficients in $\Phi_i$ and $\Psi_i$ are related as in Equation \eqref{eqn:coeff}.  Thus the intersection of these languages is a regular language ${\mathcal H}$.  As when $s$ was a type 1 generator, we conclude that ${\mathcal N}_s' \cap {\mathcal H}$ is a regular language which contains all convolutions of the form $\otimes(\sigma(g),\sigma(gs))$ where $s$.

Regardless of whether $s$ is a type 1 or type 2 generator, to complete the proof of Theorem \ref{thm:mult-lang} we must show that if ${\mathcal N}_s={\mathcal P}_s{\mathcal N}_s'$, then ${\mathcal L}_s = {\mathcal N}_s$, where ${\mathcal L}_s$ is the multiplier language for the generator $s$.   It is clear that ${\mathcal K}_2 \subset {\mathcal N}_s$.  We now prove the reverse inclusion.

Let $\otimes(p_1,p_2) \in {\mathcal P}_s$ and
$$\otimes(\Phi_1 \# \Phi_2 \# \cdots \# \Phi_d,\Psi_1 \# \Psi_2 \# \cdots \# \Psi_d) \in {\mathcal N}_s'.$$
Suppose that $g$ (resp. $h$) in $\GG$ has $\pi(g) = p_1$ (resp. $\pi(h) = p_2$) and $\sigma(g) = \Phi_1 \# \Phi_2 \# \cdots \# \Phi_d$ (resp. $\sigma(h) = \Psi_1 \# \Psi_2 \# \cdots \# \Psi_d$). We will show that $h=gs$.

We can write $g$ and $h$ in matrix form, where the entries of $p_1$ and $p_2$, respectively,  determine the exponents in the upper left entry of each matrix.  Let $p_1 = (m_1,m_2, \cdots ,m_{d-1})$ and $p_2 = (n_1,n_2, \cdots ,n_{d-1})$.  We know that these strings differ in the manner proscribed by ${\mathcal P}_s$ which corresponds to multiplication by $s$.  The upper right entry of the matrix representing $g$ is then
$\Pi_{n=1}^{d-1}(t+l_n)^{m_n}(\Phi_1 + \Phi_2 + \cdots + \Phi_d)$.  We construct an analogous polynomial using $p_2$ and  the $\Psi_n$ for the upper right entry in $h$.

Now consider the matrix for the element $gs$.  Since $\otimes(p_1,p_2) \in {\mathcal P}_s$, we must have $\pi(gs) = (n_1,n_2, \cdots ,n_d) = \pi(h)$.  If the polynomial  in the upper right entry of $gs$ is denoted $P$, use the Decomposition Lemma to write
$$\Pi_{k=1}^{d-1}(t+l_k)^{-n_k}P = P_1+P_2+ \cdots + P_d.
$$
In order to show that $h=gs$, we must verify that the polynomial entries are also the same.   To see this, first note that the values of the $\delta_n$ are encoded in the lengths of the substrings of $\sigma(g)$, $\sigma(h)$ and $\sigma(gs)$, and hence we know that the minimal degrees of the $Q_n$ (arising from $h$) and $P_n$ (arising from $gs$) are identical for all $n$.

Suppose that $s$ is a type 1 generator, and $n \neq i,d$.  We will show that $Q_n=P_n$.  From Table \ref{fig:type1chart} we see that the value of $\beta_{\delta_n}$ in $R_n$ determines the coefficient of the minimal degree term in both $Q_n$ and $P_n$, since both $\otimes(g,h)$ and $\otimes(g,gs)$ are accepted strings.  Hence the coefficient of minimal degree in these two polynomials is identical.  The following equations determine the coefficients of the polynomials $P_n$ for $n \neq i$ in increasing order of degree, namely

 \[\gamma_k =    \left\{
\begin{array}{ll}
 -C_n\gamma_{k+1} + C_n \beta_k & \text{if $ n \neq i,d$} \\
 \\
  -l_i \gamma_{k+1} + \alpha_1' \beta_{k+1} & \text{if  $n=d$}
\end{array}
\right. \]
and hence the remaining coefficients of both $Q_n$ and $P_n$ are determined by the coefficients of $R_n$ and the coefficients of higher degree in each polynomial, which are identical. Thus these two polynomials are identical.  This reasoning can be adapted both to the case $n=d$ and to type 2 generators using the following formulae from Section \ref{subsec:type2}.

 \[\gamma_k =    \left\{
\begin{array}{ll}
 -C_n\gamma_{k+1} + C_n D_n\beta_k & \text{if $ n \neq i,j,d$} \\
 \\
 -C_j\gamma_{k+1}-\beta_{k+1} & \text{if  $n=j$} \\
 \\
  -l_i \gamma_{k+1} + \sigma_0' \beta_{k} & \text{if  $n=d$}
\end{array}
\right. \]

It remains to verify that $Q_i=P_i$.  When $s$ is a type 1 generator, the coefficients of the terms  of degree less than $-1$ in both $Q_i$ and $P_i$ are simply a translate of the coefficients of $R_i$, hence identical. When $s$ is a type 2 generator, the coefficients of the terms  of degree less than $-1$ in both $Q_i$ and $P_i$ are uniquely determined by the coefficients of $R_i$, and hence identical.  Regardless of the type of generator, Equations \eqref{eqn:firstcoeff} and \eqref{eqn:firstcoefftype2} demonstrate that the initial coefficient of $P_i$ and $Q_i$ only depends on the entries of $\sigma(g)$ and hence must be identical.  Hence $h=gs$ and it follows that ${\mathcal N}_s = {\mathcal L}_s$.  This finishes the proof of Theorem \ref{thm:mult-lang} that the multiplier languages are regular for each generator $s \in S_{d,q}$, and we conclude that $\GG$ is graph automatic.

\bibliographystyle{plain}
\bibliography{refs}

\end{document}